\documentclass[a4paper,11pt]{article}
\usepackage[nointlimits]{amsmath}
\usepackage{amssymb,amsmath,amsthm}
\usepackage{amsfonts}
\usepackage{latexsym}
\usepackage{graphicx}
\usepackage{color}
\usepackage{layout}
\usepackage[english]{babel}
\numberwithin{equation}{section} 

\def\ds{\displaystyle}

\def\ZZ{\mathbb Z}
\def\NN{\mathbb N}
\newcommand{\leg}[2]{\left({#1\over #2}\right)}

\newtheorem{thm}{Theorem}[section]
\newtheorem{lem}[thm]{Lemma}
\newtheorem{cor}[thm]{Corollary}

\begin{document}

\title{\bf Congruences involving the reciprocals \\ of central binomial coefficients}

\author{{\sc Roberto Tauraso}\\
Dipartimento di Matematica\\
Universit\`a di Roma ``Tor Vergata'', Italy\\
{\tt tauraso@mat.uniroma2.it}\\
{\tt http://www.mat.uniroma2.it/$\sim$tauraso}
}

\date{}
\maketitle
\begin{abstract}
\noindent  We present several congruences modulo a power of prime $p$
concerning sums of the following type $\sum_{k=1}^{p-1}{m^k\over k^r}{2k\choose k}^{-1}$
which reveal some interesting connections with the analogous infinite series. 
\end{abstract}

\section{Introduction}
In 1979, Ap\`ery \cite{Ap:79} proved that $\zeta(3)$ was irrational starting from the identity
$$\sum_{k=1}^{\infty}{(-1)^k\over k^3}{2k \choose k}^{-1}=-{2\over 5}\, \zeta(3).$$
He also noted that
$$\sum_{k=1}^{\infty}{1\over k^2}{2k \choose k}^{-1}={1\over 3}\, \zeta(2)$$
which has been known since the nineteenth century.
Here we would like to show that the following analogous congruences hold 
for any prime $p>5$
$$\sum_{k=1}^{p-1}{(-1)^k\over k^3}{2k \choose k}^{-1}\equiv-{2\over 5}\, {H(1)\over p^2}\quad \mbox{ and }\quad
\sum_{k=1}^{p-1}{1\over k^2}{2k \choose k}^{-1}\equiv{1\over 3}\, {H(1)\over p} \pmod{p^3}$$
where $H(1)=\sum_{k=1}^{p-1}{1\over k}$ (note that $H(1)\equiv 0$ (mod $p^2$) by Wolstenholme's theorem).

After some preliminary results, in the last section we will give the proofs of the above congruences together with the {\sl dual} ones:
$$\sum_{k=1}^{p-1}{(-1)^k\over k^2}{2k \choose k}={4\over 5}\,\left({H(1)\over p}+2p\,H(3)\right)
\quad \mbox{ and }\quad
\sum_{k=1}^{p-1}{1\over k}{2k \choose k}=-{8\over 3}\,H(1)
\pmod{p^4}.$$ 

\section{Old and new results concerning multiple harmonic sums}
We define the {\it multiple harmonic sum} as
$$H(a_1,a_2,\dots,a_r;n)=
\sum_{1\leq k_1<k_2<\dots<k_r\leq n}\; {1\over k_1^{a_1}k_2^{a_2}\cdots k_r^{a_r}}$$
where $n\geq r>0$ and $(a_1,a_2,\dots,a_r)\in (\NN^*)^r$.
The values of many harmonic sums modulo a power of prime $p$ are well known and usually
they are expressed as a combination of Bernoulli numbers $B_n$.
These are the results we need later (note that we write simply $H(a_1,a_2,\dots,a_r)$ when $n=p-1$):

\begin{enumerate} 

\item[(i).] (\cite{Sunzh:00})  for any prime $p>5$
\begin{align*}
&H(1)\equiv -{1\over 2}p\,H(2)\equiv
p^2\left(2{B_{p-3}\over p-3}-{B_{2p-4}\over 2p-4}\right)  &\pmod{p^4}\\
&H(3)\equiv -{3\over 2}p\,H(4)\equiv 6p^2\,{B_{p-5}\over p-5} &\pmod{p^3}\\
&H(5)\equiv 0 \pmod{p^2}\quad \mbox{and}\quad H\left(2;{p-1\over 2}\right)
\equiv -7\,{H(1)\over p} &\pmod{p^3}
\end{align*}

\item[(ii).] (\cite{Ho:07}, \cite{Zh:06}) for $a,b>0$ and for any prime $p>a+b+1$
\begin{align*}
&H(a,b)\equiv {(-1)^b\over a+b}{a+b\choose a} \,B_{p-a-b} 
\quad \mbox{and}\quad
H(1,1,2)\equiv 0 &\pmod{p}.
\end{align*}
\end{enumerate}

This is a generalization of Wolstenholme's theorem which improves the modulo $p^4$ congruence
given in Remark 5.1 in \cite{Sunzh:00}.

\begin{thm}\label{T21} For any prime $p>5$
$$H(1)\equiv-{1\over 2}p\,H(2)-{1\over 6}p^2\,H(3)\equiv
p^2\left({B_{3p-5}\over 3p-5}-3{B_{2p-4}\over 2p-4}+3{B_{p-3}\over p-3}\right)+p^4 {B_{p-5}\over p-5}
\pmod{p^5}.
$$
\end{thm}
\begin{proof}  Let $m=\varphi(p^5)=p^4(p-1)$ then by Euler's theorem and by Faulhaber's formula
we have that for $r=1,2,3$
$$H(r)\equiv \sum_{k=1}^{p-1}k^{m-r}={B_{m-r+1}(p)-B_{m-r+1}\over m-r+1}
\pmod{p^5}.$$
Therefore
$$\sum_{r=1}^3 \alpha_r\, p^{r-1} H(r)\equiv
\sum_{r=1}^3 \alpha_r\,\sum_{k=r}^{m-1}{p^k B_{m-k}\over k-r+1}  {m-r\choose k-r}\pmod{p^5}.$$
Since $m$ is even then $B_{m-k}=0$ when $m-k>1$ and $k$ is odd.
Moreover $pB_{m-k}$ is $p$-integral, thus the sum modulo $p^5$ simplifies to
$$\sum_{r=1}^3 \alpha_r\, p^{r-1} H(r)\equiv
{p^2\over 2}(2\alpha_2-\alpha_1)B_{m-2}+
{p^4\over 4}(-6\alpha_3+4\alpha_2-\alpha_1)B_{m-4} \pmod{p^5}.$$
Hence the r.h.s. becomes zero as soon as we let $\alpha_1=1$, $\alpha_2={1\over 2}$, and $\alpha_3={1\over 6}$.
Finally we use the formulas for $H(2)$ and $H(3)$ modulo $p^4$ given in Remark 5.1 in \cite{Sunzh:00}.
\end{proof}

The next lemma allow us to expand two kinds of binomial coefficients 
as a combination of multiple harmonic sums.

\begin{lem}\label{L23} Let $n\in\NN^*$, then for $k=1,\dots,n-1$
\begin{align*}
&{n\choose k}={n\over k}{n-1\choose k-1}
=(-1)^{k-1}{n\over k}\sum_{j=0}^{k-1} (-n)^jH(\{1\}^j;k-1)\,,\\
&{n+k-1\choose k}={n\over k}{n+k-1\choose k-1}
={n\over k}\sum_{j=0}^{k-1} n^jH(\{1\}^j;k-1)
\end{align*}
\end{lem}
\begin{proof} It suffices to use the definition of binomial coefficient:
$${n-1\choose k-1}={(n-1)\cdots (n-(k-1))\over (k-1)!}=(-1)^{k-1}\prod_{j=1}^{k-1}\left(1-{n\over j}\right)
=(-1)^{k-1}\sum_{j=0}^{k-1}(-n)^jH(\{1\}^j;k-1),$$
and
$${n+k-1\choose k-1}={(n+k-1)\cdots (n+1)\over (k-1)!}
=\prod_{j=1}^{k-1}\left(1+{n\over j}\right)
=\sum_{j=0}^{k-1} n^jH(\{1\}^j;k-1).$$
\end{proof}

By (ii), we already know that $H(1,2)\equiv -H(2,1)\equiv B_{p-3}\pmod{p}$. 
Moreover, since $H(1)H(2)=H(1,2)+H(2,1)+H(3)$ then $H(1,2)\equiv -H(2,1)$ mod $p^2$.
Thanks to an identity due to Hern\'andez \cite{He:99}, we are able to disentagle $H(1,2)$ and $H(2,1)$
and prove the following congruence modulo $p^2$.

\begin{thm}\label{T22} For any prime $p>3$
$$H(1,2)\equiv -H(2,1)\equiv -3\,{H(1)\over p^2}\pmod{p^2}.$$
\end{thm}
\begin{proof} The following identity appears in \cite{He:99}: for $n\geq 1$
$$\sum_{k=1}^n {1\over k^2}=\sum_{1\leq i\leq j \leq n}{(-1)^{j-1}\over ij}{n\choose j}.$$
Letting $n=p$, by Lemma \ref{L23}, (i), and (ii) we obtain
\begin{eqnarray*}
H(2)&=&p\!\!\!\sum_{1\leq i\leq j \leq p-1}\!\!{(-1)^{j-1}\over ij^2}{p-1\choose j-1}+{H(1)\over p}
\equiv\sum_{1\leq i\leq j \leq p-1}\!\!{1-pH(1;j-1)\over ij^2}+{H(1)\over p}\\
&\equiv& pH(1,2)+pH(3)-p^2\!\!\sum_{1\leq i<j \leq p-1}\!\!{H(1;j-1)\over ij^2} -p^2H(1,3)+{H(1)\over p}\\
&\equiv& pH(1,2)+pH(3)-2p^2H(1,1,2)-p^2H(2,2) -p^2H(1,3)+{H(1)\over p} \pmod{p^3}.
\end{eqnarray*}
\end{proof}

Finally we get this extension of a result contained in \cite{Zh:07}.

\begin{thm}\label{T23} For any prime $p>5$
$${1\over 2}{2p\choose p}\equiv 1+2pH(1)+{2\over 3}p^3H(3)  \pmod{p^6}.$$
\end{thm}
\begin{proof}
Since for $n\geq 1$ (see for example \cite{Ta:09})
\begin{eqnarray*}
&&2H(\{1\}^2;n)=H^2(1;n)-H(2;n)\\
&&6H(\{1\}^3;n)=H^3(1;n)-3H(1;n)H(2;n)+2H(3;n)\\
&&24H(\{1\}^4;n)=H^4(1;n)-6H^2(1;n)H(2;n)+8H(1;n)H(3;n)+3H^2(2;n)-6H(4;n)
\end{eqnarray*}
then by (i) and by Theorem \ref{T21} we obtain
$$\begin{array}{ll}
\ds H(\{1\}^2)\equiv -H(2)\equiv {H(1)\over p}+{1\over 6}pH(3)&\pmod{p^4}\\
\ds H(\{1\}^3)\equiv {1\over 3}H(3) &\pmod{p^3}\\
\ds H(\{1\}^4)\equiv -{1\over 4}H(4)\equiv {1\over 6}{H(3)\over p}&\pmod{p^2}.
\end{array}$$
Hence by Lemma \ref{L23}
\begin{eqnarray*}
{1\over 2}{2p\choose p}&\equiv & 1-2pH(1)+4p^2H(\{1\}^2)-8p^3H(\{1\}^3)+16p^4H(\{1\}^4)\\
&\equiv& 1-2pH(1)+4pH(1)+{2\over 3}p^3H(3)-{8\over 3}p^3H(3)+{8\over 3}p^3H(3)\\
&\equiv&1+2pH(1)+{2\over 3}p^3H(3) \pmod{p^6}.
\end{eqnarray*}
\end{proof}

\section{Some preliminary results}
We consider the Lucas sequences $\{u_n(x)\}_{n\geq 0}$ and $\{v_n(x)\}_{n\geq 0}$ 
defined by these recurrence relations
\begin{align*}
&u_0(x)=0, \;u_1(x)=1, \mbox{and}\; u_{n+1}=x\,u_n(x)-u_ {n-1}(x)\mbox{ for $n>0$,}\\
&v_0(x)=2, \;v_1(x)=x, \mbox{and}\; v_{n+1}=x\,v_n(x)-v_ {n-1}(x)\mbox{ for $n>0$.}
\end{align*}
The corresponding generating functions are
$$U_x(z)={z\over z^2-xz+1}\quad\mbox{and}\quad V_x(z)={2-xz\over z^2-xz+1}.$$

Now we consider two types of sums depending on an integral parameter $m$.
Note that the factor $p$ before the sum is needed because it cancels out the other factor
$p$ at the denominator of ${2k\choose k}^{-1}$ for $k=(p+1)/2,\dots,p-1$. 

\begin{thm}\label{T31} Let $m\in\ZZ$ then for any prime $p\not=2$, 
$$p\sum_{k=1}^{p-1} {m^k\over k}{2k \choose k}^{-1}\equiv {m\, u_p(2-m)-m^p\over 2} \pmod{p^2}.$$
and
$$p\sum_{k=1}^{p-1} {m^k\over k^2}{2k \choose k}^{-1}\equiv {2-v_p(2-m)-m^p\over 2p} \pmod{p^2}.$$
\end{thm}
\begin{proof} Let $f(z)=1/(1+mz)$ then
\begin{eqnarray*}
\sum_{k=1}^n{n\choose k}{n-1+k \choose k-1}{2k \choose k}^{-1}(-m)^k \!\!\!&=&
\sum_{k=1}^n{n+k-1 \choose 2k-1}(-m)^k=[z^n]f\left(z\over (1-z)^2\right)\\
&=&-{m\over 2}[z^n]U_{2-m}(z).
\end{eqnarray*}
Let $n=p$, then
$$\sum_{k=1}^{p-1}{p\choose k}{p-1+k \choose k-1}{2k \choose k}^{-1}(-m)^k=-{mu_p(2-m)-m^p\over 2}.$$
On the other hand, by Lemma \ref{L23} the l.h.s. is congruent modulo $p^2$ to
$$
-p\sum_{k=1}^{p-1}{m^k\over k}\left(1-pH(1;k-1)\right)\left(1+ p H(1;k-1)\right){2k \choose k}^{-1}
\equiv
-p\sum_{k=1}^{p-1}{m^k\over k}{2k \choose k}^{-1}\pmod{p^2}.
$$
As regards the second congruence, consider 
\begin{eqnarray*}
\sum_{k=0}^n{n\choose k}{n-1+k \choose k}{2k \choose k}^{-1}(-m)^k \!\!\!&=&
\sum_{k=0}^n\left({n+k \choose 2k}-{1\over 2}{n+k-1 \choose 2k-1}\right)(-m)^k\\
&=&[z^n]{1\over 1-z}f\left(z\over (1-z)^2\right)-{1\over 2}f\left(z\over (1-z)^2\right)\\
&=&{1\over 2}[z^n]V_{2-m}(z).
\end{eqnarray*} 
Let $n=p$, then
$${1\over p}\sum_{k=1}^{p-1}{p\choose k}{p-1+k \choose k}{2k \choose k}^{-1}(-m)^k=-{2-v_p(2-m)-m^p\over 2p}.$$
As before, by Lemma \ref{L23} the l.h.s. is congruent modulo $p^2$ to
$$
-p\sum_{k=1}^{p-1}{m^k\over k^2}\left(1-pH(1;k-1)\right)\left(1+ p H(1;k-1)\right){2k \choose k}^{-1}
\equiv
-p\sum_{k=1}^{p-1}{m^k\over k^2}{2k \choose k}^{-1}\pmod{p^2}.
$$
\end{proof}

For some values of the parameter $x=2-m$, Lucas sequences have specific names.
Here it is a short list of examples which follow straightforwardly from 
the previous theorem. 

\begin{cor}\label{C32} For any prime $p\not=2$, the following congruences hold modulo $p^2$
$$\begin{array}{lll}
\ds p\sum_{k=1}^{p-1}{1\over k^2}{2k \choose k}^{-1}\equiv {\delta_{p,3}\over 2}
&\hspace{10mm}&\ds p\sum_{k=1}^{p-1}{1\over k}{2k \choose k}^{-1}\equiv {\leg{p}{3}-1\over 2} \\
\ds p\sum_{k=1}^{p-1}{(-1)^k\over k^2}{2k \choose k}^{-1}\equiv {1-L_{p}^2\over 2p} 
&&\ds p\sum_{k=1}^{p-1}{(-1)^k\over k}{2k \choose k}^{-1}\equiv {1-L_pF_{p}\over 2}\\
\ds p\sum_{k=1}^{p-1}{2^k\over k^2}{2k \choose k}^{-1}\equiv -q_p(2)
&&\ds p\sum_{k=1}^{p-1}{2^k\over k}{2k \choose k}^{-1}\equiv \leg{-1}{p}-1-pq_p(2)
\end{array}
$$
where $\delta_{n,k}=1$ if $n=k$ and it is $0$ otherwise, $F_n$ is the $n$-th Fibonacci number,
$L_n$ is the $n$-th Lucas number,
$\leg{\cdot}{p}$ the Legendre symbol and $q_p(a)=(a^{p-1}-1)/p$ the Fermat quotient.
\end{cor}
\begin{proof}  For $m=1$, $u_p(2-m)=\leg{p}{3}$ and $v_p(2-m)=1-3\delta_{p,3}$.
For $m=-1$, $u_p(2-m)=F_{2p}=L_p F_p$ and $v_p(2-m)=L_{2p}=L_p^2+2$
(note that $L_p\equiv 1$ and $F_p\equiv \leg{p}{5}$ mod $p$).
For $m=2$, $u_p(2-m)=\leg{-1}{p}$ and $v_p(2-m)=0$.
\end{proof}

The interested reader can compare some of the previous formulas with
the corresponding values of the infinite series (when they converge).
For example (see \cite{Le:85}) notice the {\sl golden ratio} in
$$\sum_{k=1}^{\infty}{(-1)^k\over k}{2k \choose k}^{-1}=-{2\over \sqrt{5}}\log\left({1+\sqrt{5}}\over 2\right)$$
while the analogous congruence involves Fibonacci and Lucas numbers.

By differentiating the generating functions employed in the proof of Theorem \ref{T31} one can obtain
more congruences concerning
$$p\sum_{k=1}^{p-1}{Q(k)\,m^k\over k^2}{2k \choose k}^{-1}$$
where $Q(k)$ is a polynomial. For example it is not hard to show that for any prime $p>3$
$$p\sum_{k=1}^{p-1}{1\over C_k}\equiv {2\over 3} \leg{p}{3} \pmod{p^2}$$
where $C_k={2k\choose k}/(k+1)$ is the $k$-th Catalan number.

The problem to raise the power of $k$ at the denominator seems to be much more difficult since we should
try to integrate those generating functions. In the next section we will see a remarkable example where
the power of $k$ is $3$.

Moreover the congruences established in Theorem \ref{T31} are similar to
other congruences involving the central binomial coefficients (not inverted) obtained in our joint-work
with Zhi-Wei Sun \cite{SuzwTa:09}. The next theorem gives a first explanation of this behaviour,
but it's our opinion that this relationship should be investigated further in future studies.

\begin{thm}\label{T33} Let $m,r\in\ZZ$ then for any prime $p\not 2$ such that $p$ does not divide $m$ then
$$p\sum_{k=1}^{p-1} {m^k\over k^{r}} {2k\choose k}^{-1}
\equiv {m\,(-1)^{r-1}\over 2}\sum_{k=1}^{p-1} {1\over m^{k} k^{r-1}} {2k\choose k} \pmod{p}.$$
\end{thm}
\begin{proof}
We first show that for $k=1,\dots,p-1$:
$${p\over k}{2k\choose k}^{-1}\equiv {1\over 2}{2(p-k)\choose p-k}\pmod{p}.$$
This is trivial for  $k=1,\dots,(p-1)/2$, because the l.h.s and the r.h.s have a factor $p$ at the numerator
and therefore they are both zero modulo $p$.

Assume now that $k=(p+1)/2,\dots,p-1$:
\begin{eqnarray*}
{p\over k}{2k\choose k}^{-1}&=&
{(k-1)!\over (p+(2k-p))\cdots (p+1)(p-1)\cdots(p-(p-k-1))}\\
&\equiv& {(k-1)! (-1)^k \over (2k-p)!(p-k-1)!}=(-1)^k{k-1\choose p-k-1}\pmod{p}.
\end{eqnarray*}
On the other hand 
\begin{eqnarray*}
{1\over 2}{2(p-k)\choose p-k}&=&{2p-2k-1\choose p-k-1}=
{(p-(2k+1-p))\cdots (p-(k-1))\over (p-k-1)!}\\
&\equiv&{(-1)^k (k-1)\cdots (2k+1-p))\over (p-k-1)!}=(-1)^k{k-1\choose p-k-1}\pmod{p}.
\end{eqnarray*}
By summing over $k$ and by Euler's theorem we get
\begin{eqnarray*}
p\sum_{k=1}^{p-1} {m^k\over k^r} {2k\choose k}^{-1}
&\equiv& {1\over 2}\sum_{k=1}^{p-1} {m^k\over k^{r-1}} {2(p-k)\choose p-k}=
{1\over 2}\sum_{k=1}^{p-1} {m^{p-k}\over (p-k)^{r-1}} {2k\choose k}\\
&\equiv&{m\,(-1)^{r-1}\over 2}\sum_{k=1}^{p-1} {1\over m^{k}k^{r-1}} {2k\choose k} \pmod{p}.
\end{eqnarray*}
\end{proof}

\section{Proof of the main result and a conjecture}
We finally prove the congruences announced in the introduction.

\begin{thm}\label{T41} For any prime $p>5$
$$\sum_{k=1}^{p-1} {1\over k^2}{2k \choose k}^{-1}\equiv {1\over 3}{H(1)\over p} \pmod{p^3}.$$
and
$$\sum_{k=1}^{p-1} {(-1)^k\over k^3}{2k \choose k}^{-1}\equiv -{2\over 5}{H(1)\over p^2} \pmod{p^3}.$$
\end{thm}
\begin{proof} By Lemma \ref{L23} we have that
\begin{eqnarray*}
p{p-1+k\choose k}^{-1}{p-1\choose k}^{-1}\!\!\!\!\!\!\!
&\equiv& (-1)^k k\left(\sum_{j=0}^3 p^jH(\{1\}^j;k-1)\right)^{-1}\left(\sum_{j=0}^3 (-p)^jH(\{1\}^j;k)\right)^{-1}\\
&\equiv& (-1)^k k\left(1+{p\over k}+p^2 H(2;k)+{p^3\over k}H(2;k)\right) \pmod{p^4}
\end{eqnarray*}
where in the second step  we used the relations 
$$\mbox{$H(\{1\}^j;k)=H(\{1\}^j;k-1)+{1\over k}H(\{1\}^{j-1};k-1)$ for $j\geq 1$,
and $H(1;k)^2=2H(1,1;k)+H(2;k)$.}$$
The rest of the proof depends heavily on two curious identities which play an important role in Ap\`ery's work
(see \cite{Ap:79} and \cite{VPo:79}). The first one is:
$$\sum_{k=1}^{n} {1\over k^2}{2k \choose k}^{-1}=
-{2\over 3}\sum_{k=1}^n {(-1)^k\over k^2}-{(-1)^{n}\over 3}\sum_{k=1}^n {(-1)^k\over k^2}{n+k\choose k}^{-1}{n\choose k}^{-1}.$$
Letting $n=p-1$,  by (i) we obtain
$$\sum_{k=1}^{p-1} {(-1)^k\over k^2}=-H(2)+{1\over 2}H\left(2;{p-1\over 2}\right)
\equiv 2{H(1)\over p}+{1\over 2}\left(-7{H(1)\over p}\right)\equiv -{3\over 2}\,{H(1)\over p} \pmod{p^3}.$$
By (i), (ii), and Theorem \ref{T22} we get 
\begin{eqnarray*}
\sum_{k=1}^{p-1} {(-1)^k\over k^2}{p-1+k\choose k}^{-1}{p-1\choose k}^{-1}\!\!\!\!\!\!\!
&\equiv&
{1\over p}\sum_{k=1}^{p-1}\left({1\over k}+{p\over k^2}+{p^2\over k} H(2;k)+{p^3\over k^2}H(2;k)\right)\\
&\equiv& {H(1)\over p} +H(2)+p(H(2,1)+H(3))+p^2(H(2,2)+H(4))\\
&\equiv& (1-2+3){H(1)\over p}\equiv 2{H(1)\over p} \pmod{p^3}.
\end{eqnarray*}
Hence
$$\sum_{k=1}^{p-1} {1\over k^2}{2k \choose k}^{-1}\equiv -{2\over 3}\left( -{3\over 2}\,{H(1)\over p}\right)
-{1\over 3}\left(2\,{H(1)\over p}\right)\equiv {1\over 3}{H(1)\over p} \pmod{p^3}.$$
The second identity is:
$$\sum_{k=1}^{n} {(-1)^k\over k^3}{2k \choose k}^{-1}=
-{2\over 5}\sum_{k=1}^n {1\over k^3}+{1\over 5}\sum_{k=1}^n {(-1)^k\over k^3}{n+k\choose k}^{-1}{n\choose k}^{-1}.$$
Let $n=p-1$ then since $2H(2,2)=H(2)^2-H(4)\equiv -H(4)$ mod $p^2$ we have that
\begin{eqnarray*}
\sum_{k=1}^{p-1} {(-1)^k\over k^3}{p-1+k\choose k}^{-1}{p-1\choose k}^{-1}\!\!\!\!\!\!\!
&\equiv&
{1\over p}\sum_{k=1}^{p-1}\left({1\over k^2}+{p\over k^3}+{p^2\over k^2} H(2;k)+{p^3\over k^3}H(2;k)\right)\\
&\equiv& {H(2)\over p} +H(3)+p(H(2,2)+H(4))+p^2(H(2,3)+H(5))\\
&\equiv& {H(2)\over p} +H(3)+{p\over 2}\left({4p\over 5}B_{p-5}\right)+p^2(-2B_{p-5}+0)\\
&\equiv& {H(2)\over p} +{7\over 3}H(3)\pmod{p^3}.
\end{eqnarray*}
Hence by Theorem \ref{T21}
\begin{eqnarray*}
\sum_{k=1}^{p-1}{(-1)^k\over k^3}{2k \choose k}^{-1}
&\equiv&-{2\over 5}H(3)+{1\over 5}\left({H(2)\over p} +{7\over 3}H(3)\right)\\
&\equiv&-{2\over 5}\left(-{1\over 2}{H(2)\over p}-{1\over 6}H(3)\right)\equiv -{2\over 5} {H(1)\over p^2}\pmod{p^3}.
\end{eqnarray*}
\end{proof}

The duality established in Theorem \ref{T33} suggests some  analogous result for 
the sums involving the central binomial coefficients (not reversed).
By an identity due to Staver \cite{St:47} it is possible to prove (see \cite{SuzwTa:09} and \cite{Ta:09}) that
for any prime $p>3$
$$\sum_{k=1}^{p-1}{1\over k}{2k \choose k}=-{8\over 3}\,H(1) \pmod{p^4}.$$
Moreover this other congruence holds.

\begin{thm}\label{T42} For any prime $p>5$
$$\sum_{k=1}^{p-1}{(-1)^k\over k^2}{2k \choose k}={4\over 5}\,\left({H(1)\over p}+2p\,H(3)\right) \pmod{p^4}.$$ 
\end{thm}
\begin{proof} The following identity was conjectured in \cite{BB:97} and then it has been proved \cite{AG:99}:
for $n\geq 1$
$${5\over 2}\sum_{k=1}^n {2k\choose k}{k^2\over 4n^4+k^4}\prod_{j=1}^{k-1}{n^4-j^4\over 4n^4+j^4}={1\over n^2}.$$
Let $n=p$ then for $1\leq k<p$
\begin{eqnarray*}
\prod_{j=1}^{k-1}{p^4-j^4\over 4p^4+j^4}&=&
(-1)^{k-1}\prod_{j=1}^{k-1}{1-(p/j)^4\over 1+4(p/j)^4}\equiv
(-1)^{k-1}\prod_{j=1}^{k-1}\left(1-{p^4\over j^4}\right)\left(1-4{p^4\over j^4}\right)\\
&\equiv &(-1)^{k-1}\prod_{j=1}^{k-1}\left(1-5{p^4\over j^4}\right)
\equiv(-1)^{k-1}\left(1-5p^4 H(4;k-1)\right)\pmod{p^8}.
\end{eqnarray*}
Hence by (i) and by Theorem \ref{T23} 
\begin{eqnarray*}
\sum_{k=1}^{p-1}{(-1)^k\over k^2}{2k \choose k}&\equiv&
{2\over 5p^2}\left({1\over 2}{2p\choose p}\left(1-5p^4 H(4)\right)-1\right)\\
&\equiv& {2\over 5p^2}\left(\left(1+2pH(1)+{2\over 3}p^3H(3)\right)\left(1+{10\over 3}p^4 H(3)\right)-1\right)\\
&\equiv& {4\over 5}\,\left({H(1)\over p}+2p\,H(3)\right)
\pmod{p^4}.
\end{eqnarray*}
\end{proof}


\end{document}